\documentclass[twocolumn,11pt]{article}
\usepackage{times,amsmath,amsthm,graphicx,hyperref,float,xcolor}
\newtheorem{definition}{Definition}

\newtheorem{lemma}[definition]{Lemma}
\newtheorem{theorem}[definition]{Theorem}

\newtheorem{example}[definition]{Example}
\newcommand{\R}{{\rm I}\!{\rm R}} 
\newcommand{\N}{{\rm I}\!{\rm N}}

\begin{document}
\global\def\refname{{\normalsize \it References:}}
\baselineskip 12.5pt
%
%
%
\title{\LARGE \bf Heat transfer process with solid-solid interface:\\ Analytical and numerical solutions}

\date{}

\author{\hspace*{-10pt}
\begin{minipage}[t]{2.3in} \normalsize \baselineskip 12.5pt
\centerline{DIANA RUBIO}
\centerline{Univ. Nacional de San Mart\'in}
\centerline{Escuela de Ciencia y Tecnolog\'ia}
\centerline{Centro de Matem\'atica Aplicada}
\centerline{ITECA (CONICET-UNSAM)}
\centerline{25 de Mayo y Francia, San Mart\'in}
\centerline{ARGENTINA}
\end{minipage} \kern 0in
\begin{minipage}[t]{2.3in} \normalsize \baselineskip 12.5pt
\centerline{DOMINGO A. TARZIA}
\centerline{Universidad  Austral}
\centerline{FCE, Departamento de Matem\'atica} \centerline{Paraguay 1950, Rosario}
\centerline{and CONICET}
\centerline{ Godoy Cruz 2290, CABA,}
\centerline{ARGENTINA}
\end{minipage}
\begin{minipage}[t]{2.3in} \normalsize \baselineskip 12.5pt
\centerline{GUILLERMO F. UMBRICHT}
\centerline{Univ.~Nac. de Gral.~Sarmiento}
\centerline{Instituto de Ciencias}
\centerline{Instituto del Desarrollo Humano}
\centerline{J. M. Gutiérrez 1150}
\centerline{Los Polvorines}
\centerline{ARGENTINA}
\end{minipage}%
%
%
\\ \\ \hspace*{-10pt}
\begin{minipage}[b]{6.9in} \normalsize
\baselineskip 12.5pt {\it Abstract:}
This work is aimed at the study and analysis of the heat transport on a metal bar of length $L$ with a solid-solid interface. The process is assumed to be developed along one direction, across two homogeneous and isotropic materials. Analytical and numerical solutions are obtained under continuity conditions at the interface, that is a perfect assembly. The lateral side is assumed to be isolated and a constant thermal source is located at the left-boundary while the right-end stays free allowing the heat to transfer to the surrounding fluid by a convective process. The differences between the analytic solution and temperature measurements at any point on the right would indicate the presence of discontinuities. The greater these differences, the greater the discontinuity in the interface due to thermal resistances, providing a measure of its propagation from the interface and they could be modeled as temperature perturbations. The problem of interest may be described by a parabolic equation with initial, interface and boundary conditions, where the thermal properties, the conductivity and diffusivity coefficients, are piecewise constant functions. The analytic solution is derived by using Fourier methods. Special attention is given to the Sturm-Liouville problem that arises when deriving the solution, since a complicated eigenvalue equation must to be solved. Numerical simulations are conducted by using finite difference schemes where its convergence and stability properties are discussed along with physical interpretations of the results.
\\ [4mm] {\it Key--Words:}
Heat equation, solid-solid interface, eigenvalues problems, mathematical modeling.
\end{minipage}
\vspace{-10pt}}

\maketitle

\thispagestyle{empty} \pagestyle{empty}
%
%
\section{Introduction}
\label{S1} \vspace{-4pt}

Heat transfer problems in multilayer or solid-solid interface materials have been  arisen in a several applications in science and engineering  \cite{Chung2001}. Direct applications can be found in the industry \cite{Holler2020}, including 
metallurgical \cite{Ma2010}, aerospace \cite{Barturkin2005}, technological and  electronic \cite{Cahill2003} and aviation \cite{Ward2006}. A large  number of articles are devote to the  study  of thermal, electromagnetic and/or optical properties  of composed materials, among them  \cite{Cahill2003}-\cite{Chung2001},
\cite{{Hristov2012}},  \cite{Prabhu2002}-\cite{Stevens2007}, \cite{Volz2000}-\cite{Zeng2021}. 
These types of problems are generally approached experimentally  or through numerical simulations. Few articles are found in the literature that focus on mathematical models and analytical descriptions of the thermal process, as in \cite{Holler2020}, where the model is described. In \cite{Umbricht2020MEP}, \cite{Umbricht2021IJHT} the problem is approached analytically for the steady-state. On the other hand, the evolutionary state of the interface problem is studied in \cite{Hristov2012} for a solid material of infinite length.

This work focus on the analytical  solution to a heat transfer problem that it is assumed to occur along a bar composed by two different materials with continuity conditions at the solid-solid interface.
A thermal source is imposed at the left boundary ($x=0$) while free convection is assumed at the right side ($x=L$). To the best of authors' knowledge, the analytical solution to this problem is not published. In \cite{Ozisik1993}, the problem is stated with an  strategy for solving the equation but is it not explicitly solved.
The solution to the perfectly assembly solid-solid interface problem is important  since the differences with observed data it would provide a measure of the discontinuities due to roughness and tension between the materials. 
 
Here, an approach is presented for solving the problem analytically where the solution is obtained as a combination of the steady-state solution and a transient term, where the latter one is calculated using Fourier techniques.
This manner to present the solution is useful to better understand the physical transient behavior.

As in the case of a homogeneous bar, a Sturm-Liouville (S-L) eigenvalue problem arises. Finding its solution is complicated since the coefficients of the equation are not constant but depend on the thermal parameters of the materials involved. The existence of an infinite number of solutions to the S-L equation is demonstrated and an illustrative example is included. This is the most important result of this work.

Numerical simulations of the temperature  profile are conducted using a finite difference scheme of second order centered in space and first order forward in time.  The convergence and stability properties are discussed along with physical interpretations of the results. Analytical and numerical solutions to this problem are useful to predict temperatures profiles under different situations assuming perfect assembly between materials and hence, to detect discontinuities at the interface. 

In Section \ref{framework}, the equations used to describe the process is presented. Section \ref{steady} is aimed to the steady-state heat transfer problem associated to the one of interest. The corresponding transient problem is addressed in Section \ref{transient}, where the eigenvalue problem and the analytical solution  is obtained. In Section \ref{numerical}, some numerical examples of the temperature profile for the discretized equation are included. Finally, conclusions and future worksare discussed.
\section{Mathematical Framework}
\label{framework} \vspace{-4pt}

Consider a unidimensional heat transfer process on a material, which is modeled as a bar whose lateral surface
is totally isolated, and it is made up of two
consecutive sections of different, perfectly assembly, isotropic and homogeneous materials.
This problem can be described by coupled 
parabolic equations with interface, initial 
and boundary conditions. At the left-boundary of 
the bar, a constant thermal source is assumed
while the right-end is free allowing the 
convection process to occur (see Figure \ref{barscheme}).
\begin{figure}
\begin{center}
\includegraphics[scale=0.15]{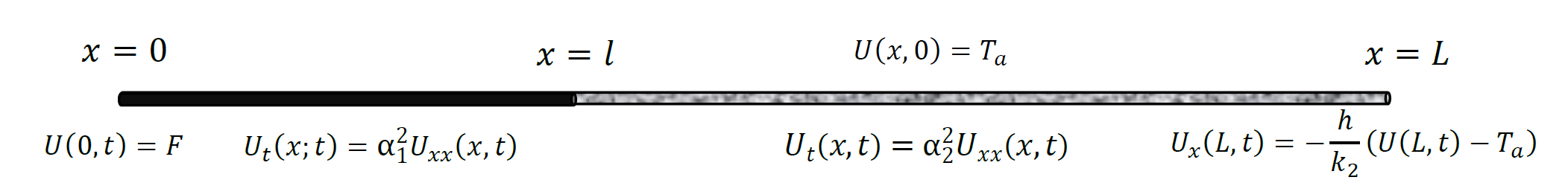}
\vspace{-0.5cm}
\caption{
Heat conduction
problem with interface.}
\label{barscheme}
\end{center}
\end{figure}

The system to be solved is given by the heat equations
\begin{eqnarray}
\label{ec1}
U_t(x,t)&=&\alpha_{1}^{2} U_{xx}(x,t), \,\,\,
0<x<l , \\
\label{ec2}
U_t(x,t)&=&\alpha_{2}^{2} U_{xx}(x,t), \,\,\,  l<x<L, 
\end{eqnarray}
for $t>0$, with initial temperature
\begin{equation}
\label{initialc}
U(x,0)=T_a,   \qquad \quad   0<x<L,
\end{equation}
and boundary conditions 
\begin{eqnarray}
\label{bc}
U(0,t)&=&F, \qquad  \qquad \qquad \quad \, t>0,\\
k_2\, U_x(L,t)&=&-h\,(U(L,t)-T_a),  \, t>0,
\end{eqnarray}
where $L$ represents the length of the bar, $T_a$ the temperature of the surrounded fluid,
$F$ denotes the temperature at $x=0$, $l$ the interface position  ($0<l<L$) and $h$ denotes the heat transfer coefficient due to convection at $x=L$. The 
coefficients $\alpha_1^2$, $k_1$ and $\alpha_2^2$, $k_2$ represent the diffusivity and the thermal conductivity for the materials at the left and 
right side of the bar, respectively. For two perfectly assembled homogeneous materials, continuity conditions are given at the interface position $ x = l $, that is,
\begin{eqnarray}
\label{interface1}
\displaystyle \lim_{x \to l^-} U(x,t)&=&\lim_{x \to l^+} U(x,t),\\
\label{interface2}
\displaystyle \lim_{x \to l^-} k_1 U_x(x,t)&=&\lim_{x \to l^+} k_2 U_x(x,t),
\end{eqnarray}
for $t>0$.
From now on, for simplicity we assume that
\begin{equation}
\label{FmayorqueTa}
F>T_a.                 
\end{equation}
\section{The steady-state  problem}
\label{steady}\vspace{-4pt}

The steady-state problem corresponding to the initial and boundary problem with interface  \eqref{ec1}-\eqref{interface2} is given by the following equations
\begin{eqnarray}
\label{eqest1}
U_{xx}^S (x)&=&0, \qquad  0<x<l,\\
\label{eqest2}
U_{xx}^S (x)&=&0,  \qquad l<x<L,\\
\label{lbc}
U^S (0)&=&F,   \hspace*{5cm}        \\
\label{rbc}
-k_2 U_x^S (L) &=& h (U^S (L)-T_a ), \hspace*{2,5cm}   \\
\label{ifbc1}
U^S(l^-)&=&U^S(l^+),\hspace*{3.8cm} \\
\label{ifbc2}
k_1 U_x^S(l^-) &=& k_2 U_x^S(l^+), \hspace*{3cm} 
\end{eqnarray}
where $U^S(l^-)$ and $U^S(l^+)$  denote $\displaystyle \lim_{x \to l^-}
U^S(x)$ and $\displaystyle \lim_{x \to l^+} U^S(x)$, respectively.

\begin{lemma}
The solution to the steady-state problem \eqref{eqest1}-\eqref{ifbc2} 
is given by the following expression:
\begin{equation}
\label{Uestacionario_mu}
U^S(x)=
\begin{cases}
F - Q \mu \frac{1}{k_1}\, x, &  0 \leq x \leq l, \\
F - Q  \mu \left(\frac{1}{k_2}(x-l) + \frac{l}{k_1} \right), &  l<x \leq L, 
\end{cases}
\end{equation}
where $Q=(F-T_a) h $, and $\mu$  is the dimensionless coefficient 
\begin{equation}
\label{mu}
\mu =\frac{1}{1+\frac{hL}{k_2}
+\left( \frac{1}{k_1} -\frac{1}{k_2} \right) hl }= \frac{k_1 
k_2}{D},
\end{equation}
being $D=k_1 k_2 +k_1 hL + (k_2 - k_1) hl$, $k_1, k_2, h, l, L, T_a, F$  positive constants, $L>l>0$.
\end{lemma}
\begin{proof}
Equations \eqref{eqest1}-\eqref{eqest2} imply that the solution is a piecewise linear function. Imposing the boundary and interface conditions \eqref{lbc}-\eqref{ifbc2} it follows that, after algebraic computations, the solution can be written as
\begin{equation}
\label{Uestacionario}
U^S(x)=
\begin{cases}
\displaystyle
F - \frac{Q k_2}{D} x,\qquad \qquad \qquad\quad0 \leq x \leq l,\\
\\
\displaystyle
F - \frac{Q k_1  
k_2}{D} \left(\frac{x-l}{k_2} + \frac{l}{k_1} \right), \,l<x \leq L.
\end{cases}
\end{equation}
By using the dimensionless coefficient $\mu$ defined in \eqref{mu}, the expression \eqref{Uestacionario_mu} is obtained.
\end{proof}
This section is included for the sake of completeness and no much detail or discussion is given here. In  \cite{Umbricht2020}, \cite{Umbricht2020MEP}, \cite{Umbricht2021IJHT}  an equivalent expression can be found for the solution to \eqref{eqest1}-\eqref{ifbc2}  and its  consistency with the corresponding one for an homogeneous bar with the same  boundary conditions.
\begin{example}
Consider the problem described by the equations 
\eqref{eqest1}-\eqref{ifbc2} with $L = 1m$, $T_a=25^\circ C$, $h=10\, W⁄(m^2 {}^\circ C )$ and $F=100^\circ C$.
\end{example}
Figure \ref{fig:Us} shows the spatial profile of temperatures for different materials and different interface points. 
It can be seen that the solution is piecewise linear and, since the thermal source is higher than the room temperature, the temperature decreases as a function of the distance from the source location. The less conductive materials leads to a greater decrease in temperature.
\begin{figure}[H]
\label{fig:stationary}
\begin{center}
\includegraphics[scale=1.]{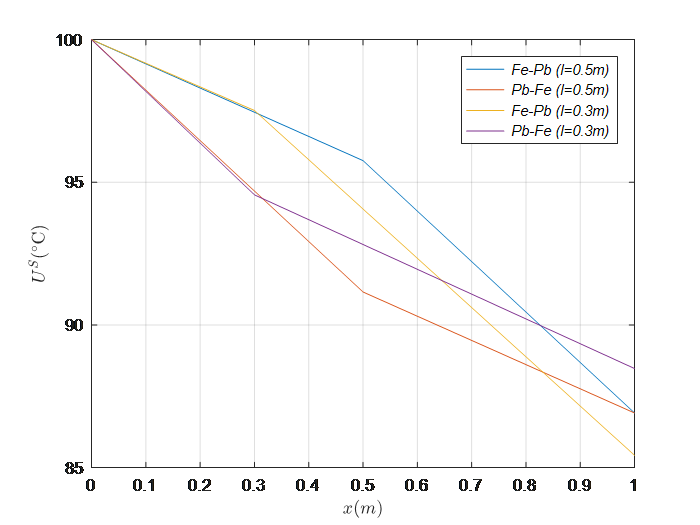} \\
\includegraphics[scale=1.]{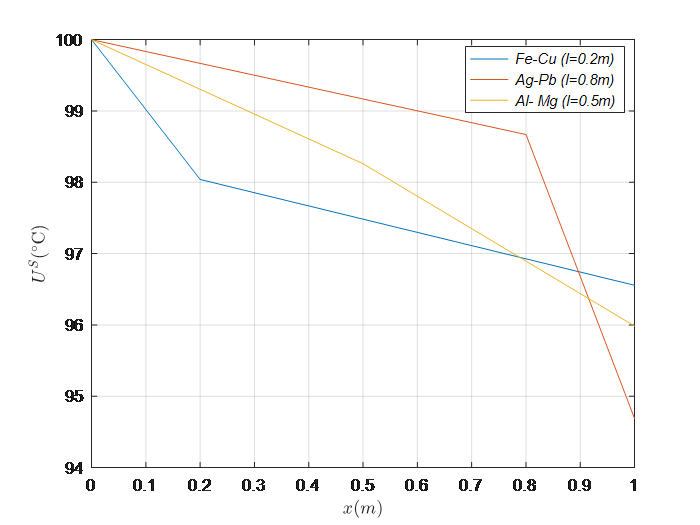}
 \end{center}
\vspace{-0.5cm}
\caption{\label{fig:Us}Temperature profiles for different contact point positions (left) and different  materials (right).}
 \end{figure}
The plots on top show different situations for a Fe-Pb or a Pb-Fe bar.
 It can be observed that in the case of  Fe-Pb,  higher temperature values are achieved for  $x<l=L/2$. This is consistent with the analytical solution since, in this case, it results
\begin{equation}
\label{Ul_half}
   U^s(l) = F - \frac{F-T_a}{ (\frac{k_1}{k_2}+1)+\frac{2k_1}{hL}}.
\end{equation}
Then, for the same pair of materials, the  temperature  values at $x=l$ are greater when the more conductive  material occupies the left half of the bar (i.e., $k_1>k_2$).
On the other hand, if  $l=L/2$ from \eqref{Uestacionario_mu} it follows that
\begin{equation}
U^s(L) = F - \frac{F-T_a}{1+\frac{2k_1 k_2}{ h L(k_1 + k_2)}}.    
\end{equation}
Then, the temperature value at $x=L$ depends on $k_1+k_2$ and $k_1k_2$, hence the relative location of the two materials to the left or right (i.e. Fe-Pb and Pb-Fe) does not influence the temperature value $U(L)$ at the right edge (see also \cite{Umbricht2020MEP}).

At the bottom of  Figure \ref{fig:Us}, the temperature profiles for different interface locations  and different material pairs are  shown. The materials were chosen so that their thermal conductivities satisfy different relationships that are reflected in the slopes of the lines. 
For Fe-Cu: $k_1<k_2$, Al-Mg: $ k_1 \simeq k_2$  (thermally similar), Ag-Pb: $k_1>k_2$ (see Table \ref{thermal_prop}).
\begin{table}
\begin{center}
\caption{\label{thermal_prop} Thermal properties of different materials.}
\begin{tabular}{lccc}
\hline
Material (Symbol)	& $k (W/m ^\circ C)$ &	$\alpha^2 \times 10^4 
(m^2/s)$ \\ 
\hline
Lead	(Pb) & 35 & 0.23673\\
Nickel	(Ni) & 70 & 0.22660\\
Iron	(Fe) & 73 & 0.20451\\
Magnesium (Mg) & 156 & 0.88300\\
Aluminium (Al) & 204 & 0.84010\\
Cupper (Cu) & 386 & 1.12530\\
Silver (Ag) & 419 & 1.70140\\
\hline
\end{tabular}
\end{center}
\vspace{-0.5cm}
\end{table}
\section{The  transient problem}
\label{transient}\vspace{-4pt}

In order to solve the problem \eqref{ec1}-\eqref{interface2}, we consider
\begin{equation}
\label{Usuma}
U(x,t)=U^s(x)+\varphi(x,t), \, 0\leq x \leq L, \, t\geq 0,
\end{equation}
where $U^s (x)$ is given by \eqref{Uestacionario} or 
\eqref{Uestacionario_mu}-\eqref{mu} and $\varphi(x,t)$ satisfies the following initial and boundary problem with interface for  $t>0$
\begin{eqnarray}
\label{fi:ec1}
 \varphi_t (x,t)&=&\alpha_1^2\,  \varphi_{xx} (x,t), \,\, 0 < x < l, \qquad\\
\label{fi:ec2}
\varphi_t (x,t)&=&\alpha_2^2\,  \varphi_{xx} (x,t),  \,\, l<x<L, \qquad\\
\label{fi:cit0}
 \varphi (x,0)&=& T_a-U^s (x),  \,\,  0 < x < L, \quad\\
\label{fi:cbx0}
 \varphi (0,t)&=&0, \\
\label{fi:cbxL}
 - k_2 \varphi_x (L,t)&=&h \varphi(L,t), \\         \label{fi:itf1}
\varphi(l^-,t)&=&\varphi(l^+,t), \\        
\label{fi:itf2}
k_1 \varphi_x (l^-,t)&=&k_2 \varphi_x (l^+,t).  
\end{eqnarray}
By using this representation, the transient terms can be viewed as " perturbations" to the steady-state.

The standard procedure of separation of variables is used to find $\varphi (x,t)$. Assuming the existence of $X(x)$ and $T(t)$ that satisfy, for $t>0$,
\begin{equation}
\label{fiXT}
\varphi(x,t)= \begin{cases}
X_1 (x).T(t),  &  0 \leq x \leq l,\qquad\\
X_2 (x).T(t),  &  l<x \leq L,  
\end{cases}
\end{equation}
and the following equations and conditions are obtained:
\begin{eqnarray}
X_1''(x)- \xi_1 X_1 (x)&=& 0,  \qquad 0 < x < l, \qquad \\
X_2'' (x)-\xi_2 X_2 (x)&=&0, \qquad   l<x < L,          \\
T' (t)&=&\xi_1 \alpha_1 T(t) \nonumber\\
      &=&\xi_2 \alpha_2 T(t), \,      t>0,\\
X_1 (0)&=&0,     \\  
k_2 X_2' (L)+h X_2 (L)&=&0, \qquad \\
X_1 (l^- )  &=&X_2 (l^+ ),      \\  
k_1 X_1'(l^- )  &=&k_2 X_2'(l^+ ). 
\end{eqnarray}
A solution to the above eigenvalue problem exists provided that $\xi_i=-\lambda_i^2 < 0$, and it follows that 
\begin{eqnarray}
X_1 (x) &=& A_1  \sin (\lambda_1 x),     \\  
X_2(x)&=&A_2  \sin (\lambda_2 x)  +B_2 \cos(\lambda_2 x),  \qquad  \\ 
\label{T1T2}
T(t) &=& C_1 e^{-\lambda_1 \alpha_1^2 t}=C_2 e^{-\lambda_2 ^2 \alpha_2^2 t},  \end{eqnarray}
where 
\begin{equation}
\label{alfa}
\lambda_1= \alpha \lambda_2, \qquad 
    \alpha =\sqrt{\displaystyle  \frac{\alpha_2^2}{\alpha_1^2}}=\frac{\alpha_2}{\alpha_1}.
\end{equation}
From now on, we denote $\lambda= \lambda_2$ and, without loss of generality, it is 
assumed that $A_1=1$.  Setting $A=A_2,B=B_2$ and $C=C_2$ we have
\begin{eqnarray}
\label{X1}
&X_1 (x) = \sin (\alpha \lambda x),  \hspace{2cm}    \\  \label{X2}                 
& X_2 (x)=A \sin (\lambda x) +B \cos(\lambda x), \\ 
\label{T}
&T(t)=Ce^{-\lambda^2 \alpha_2^2 t},   \hspace{2cm} 
\end{eqnarray} 
 where $\lambda >0$ must satisfy the eigenvalue equation
 \begin{equation}
\tan(xL)=\frac{k_2 Ax+h B}{k_2 Bx-Ah}, \hspace{1,5cm} x>0,
 \end{equation}
 or equivalenty,
 \begin{equation}
 \label{eig1}
-\frac{k_2}{h}x=\frac{B+A \tan(xL)}{A-B\tan(xL)}, \hspace{1cm} x>0.
 \end{equation}
 From the two interface conditions, and letting	
 \begin{equation}
 \label{k}
k=\frac{k_1}{k_2},
 \end{equation}
it follows that
\begin{eqnarray}
\label{A} 
 &A=k \alpha \cos(\alpha l x) \cos(lx)+\sin(\alpha lx) \sin(lx),& \,\,\,\\ 
 \label{B}
 &B=\sin(\alpha lx) \cos(lx)-l \alpha \cos(\alpha lx) \sin(lx).&\,
 \end{eqnarray}
Replacing \eqref{A}-\eqref{B} in equation \eqref{eig1}, by algebraic computation the eigenvalue equation \eqref{eig1} can be written as
\begin{equation}
\label{eig}
-\frac{k_2}{h} x=\frac{\tan (\alpha lx) + k \alpha \tan((L-l)x)}{k \alpha -\tan (\alpha lx) \tan((L-l)x)},    \end{equation}
for $x>0$. 

Lemmas \ref{deff}- \ref{tansuma} show that the right hand side of the above equation \eqref{eig} may be written as the tangent of a sum of two functions and, therefore, it has an infinite number of essential discontinuities. These results will be used in Theorem \ref{Teoeig} to prove that there exist infinitely many solutions to \eqref{eig}. 
\begin{lemma}
\label{deff}
For  $\alpha, l, L, k >0$ with  $L>l$, the function
 $f: D \subset (0,+\infty) \rightarrow \R$ defined by
\begin{equation}
\label{f}    
    f(x)= {\rm atan} 
\left( \frac{\tan(\alpha l x)}{k \alpha} \right) +(L-l)x,
\end{equation}
  satisfies $\R^+ \subseteq Im(f)$ where 
 \begin{equation}
 D = [0, +\infty) - \{x_n, \quad n \in \N\},
 \end{equation}
being 
\begin{equation}
\label{xn}
     x_n=-\frac{\pi}{2 \alpha l} + n \,\frac{\pi}{\alpha l}.
\end{equation} 
\end{lemma}
\begin{proof}
Consider the one-sided limits at the discontinuity points $x_n$ given in \eqref{xn}. Since 
\begin{equation}
{\rm atan} \left(\frac{\tan (\alpha lx)}{k \alpha}\right) \in \left(-\frac{\pi}{2}, \frac{\pi}{2 }\right),
 \end{equation}
it results that
\begin{equation}
 \lim_{x \to x_n^-}  f(x) = \frac{\pi}{2}+(L-l) x_n,
 \end{equation}
\begin{equation}
 \lim_{x \to x_n^+}  f(x)  = -\frac{\pi}{2}+(L-l) x_n.
 \end{equation}
Therefore, 
\begin{equation}
\label{step}
 \lim_{x \to x_n^-}  f(x) > \lim_{x \to x_n^+}  f(x)     
\end{equation}
and since for $x\in D$
\begin{equation}
\label{df}
  f'(x) = \frac{l}{k\left[1+ \left(\frac{\tan (\alpha lx)}{k \alpha}\right)^2\right] \cos^2 (\alpha lx)}   + L-l>0,
\end{equation}
it follows that $f$ is increasing in the interval $(0, x_1)$ and in each interval  $( x_n, x_{n+1})$, $\forall \, n \in \N$.
On the other hand, the first term of $f$  is bounded, and  $L>l$, then  
\begin{equation}
\label{unbound}
 \lim_{x \to +\infty} f(x)=+\infty.
 \end{equation}
From \eqref{step}, \eqref{df} and \eqref{unbound} it follows that all positive real values are included in $Im(f)$, and the proof is completed.
\end{proof}
Different parameter values will produce functions  $f(x)$ defined in \eqref{f} having graphs of similar shape. 
The example bellow illustrates the behaviour  for a particular case.
\begin{example}
\label{ex2}
Consider the expression of the function $f$ given in \eqref{f} for the problem described by the equations \eqref{ec1}-\eqref{interface2}  for a bar made of iron and lead (Fe-Pb).
The particular  parameter values for this example are included in Table \ref{tab:ex2}.
\end{example}
\vspace{-0.3cm}
\begin{table}[H]
\begin{center}
\caption{Parameter values for Example \ref{ex2}.}
\begin{tabular}{lc}
Parameter	& Value\\
\hline
$L(m)$ &	5\\
$l(m)$ &	2\\
$k_1 (W/m {}^\circ C)$ &	73\\
$k_2  (W/m {}^\circ C)$ &	35\\
$\alpha_1^2 (m^2⁄s)$ 	& $0.20451\times10^{-4}$\\
$\alpha_2^2 (m^2⁄s)$	& $0.23673\times10^{-4}$\\
$h (W⁄(m^2 {}^\circ C ))$	& 10\\
\hline
\end{tabular}
    \label{tab:ex2}
\end{center}
\vspace{-0.5cm}
\end{table}
Figure \ref{fig:ex2} shows  the plots of the  piecewise continuous function $f$ given in \eqref{f} for this particular case 
 \begin{equation}
     f(x)= {\rm atan} \left( \frac{\tan( 1.85892x)}{0.44562} \right) + 3x
 \end{equation}
along with  $y=3x$ and
 \begin{equation}
 {\rm atan} \left( \frac{\tan( 
 1.85892x)}{0.44562} \right).
 \end{equation}
\begin{figure}[h]
\begin{center}
    \includegraphics[scale=0.11]{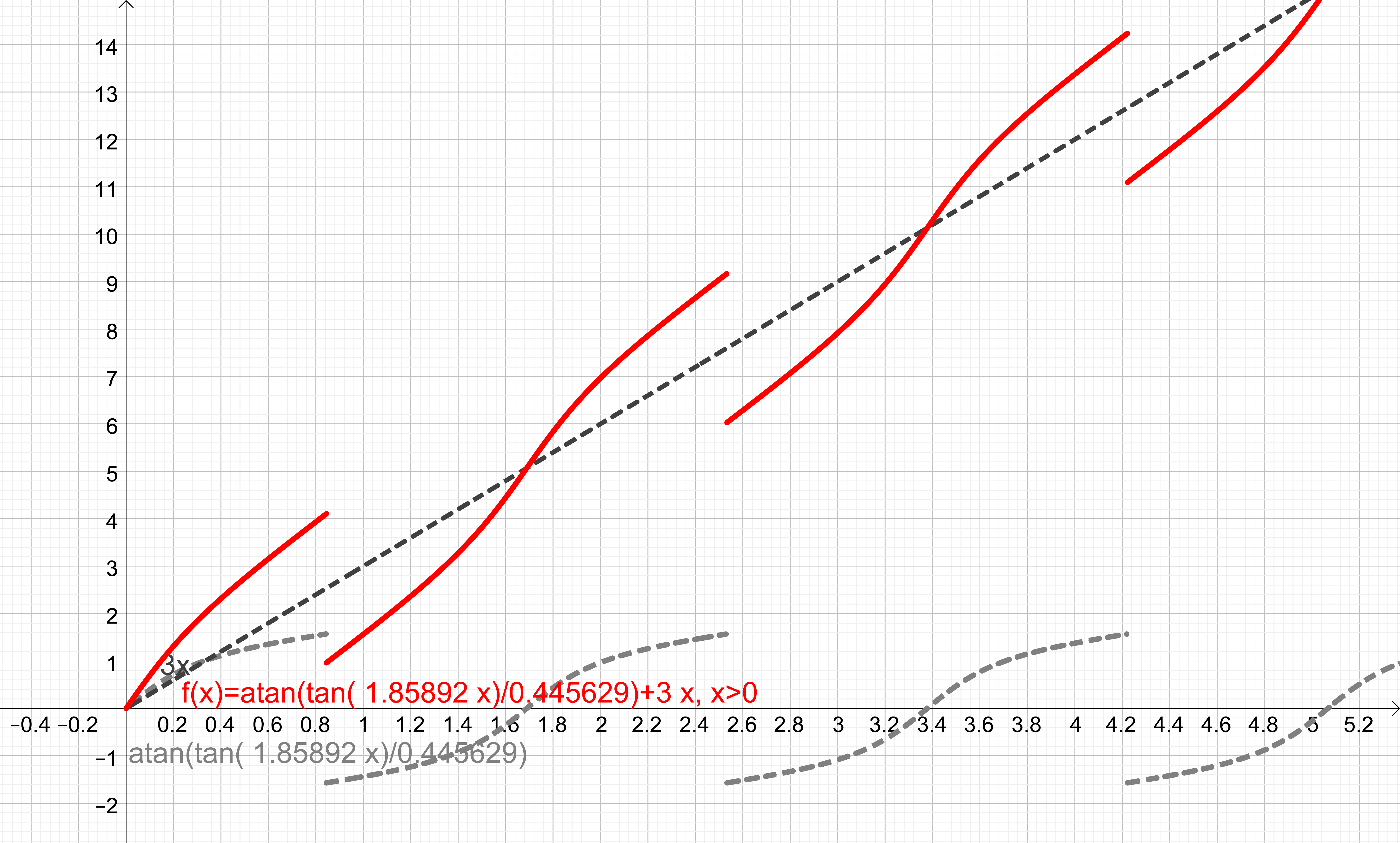}
    \caption{Red line: $f(x)= {\rm atan} \left( \frac{\tan( 1.85892x)}{0.44562} \right) + 3x$, 
    Grey dotted lines: $y=3x$ and  ${\rm atan} \left( \frac{\tan( 1.85892x)}{0.44562} \right)$ .}
    \label{fig:ex2}
    \end{center}
\end{figure}
It can be seen that, although $f$ has an infinite number of  discontinuities due to the term $\frac{\tan( 1.85892x)}{0.44562}$,  the image of the function $f$ (in red) includes all positive values. This will be crucial to prove that the equation \eqref{eig} has infinitely many solutions.
\begin{lemma} 
\label{tansuma}
Given $\alpha, L, l, k >0$, it follows that, for $x>0$ 
\begin{equation}
\label{tanf}
\frac{\tan (\alpha lx) + k \alpha \tan((L-l)x)}{k \alpha -\tan (\alpha lx) \tan((L-l)x)}=\tan(f(x)),
\end{equation}
where $f$ is defined in \eqref{f}.
\end{lemma}
\begin{proof}
Consider $f$ defined in \eqref{f}. 
By using the formula for the tangent of a sum and some algebraic computations, it follows that
\begin{align}
\label{tanfsuma}
\displaystyle
   \tan(f(x))
=& \tan \left( {\rm atan} \left( \frac{\tan( 1.85892x)}{0.44562} \right) + 3x \right)\nonumber \\
=&
\frac{\tan \left( {\rm atan}
\left( \frac{\tan(\alpha l x)}{k \alpha} \right)\right) + \tan((L-l)x)}{1-\tan \left({\rm atan} \left( \frac{\tan(\alpha l x)}{k \alpha} \right)\right) \tan((L-l)x)} \nonumber \\
=& \frac{ \frac{\tan(\alpha l x)}{k \alpha} + \tan((L-l)x)}{1- \frac{\tan(\alpha l x)}{k \alpha} \tan((L-l)x)}.
\end{align}
The equation \eqref{tanf} is obtained after multiplying the numerator and denominator in \eqref{tanfsuma} by $k\alpha$.
\end{proof}
\begin{theorem}
\label{Teoeig}
Let  $k_2, h, \alpha, l, L, k >0$, with $L>l$. The equation
\begin{equation}
\label{eig2}
-\frac{k_2}{h} x=\frac{\tan (\alpha lx) + k \alpha \tan((L-l)x)}{k \alpha -\tan (\alpha lx) \tan((L-l)x)},    
 \qquad      x>0,  
\end{equation}
has infinitely many positive solutions $0 <x_1 <x_2 \cdots <x_n < \cdots$. 
\end{theorem}
\begin{proof} 
Lemma \ref{deff} and Lemma \ref{tansuma} allow to write  
\begin{equation}
\label{proofteo1}
-\frac{k_2}{h} x=\tan (f(x)), \qquad     x>0, \end{equation}
where $f$ is defined in \eqref{f}. Lemma \ref{deff} ensures that $\R^+ \subseteq Im(f)$ implying that $\tan(f(x))$ has an infinite number of branches that intersects the line $y=-\frac{k_2}{h} x$ for $x>0$.
\end{proof}
The following 
example illustrates solutions to Equation \eqref{eig} for different setups.
\begin{example}
\label{ex3}
As for the previous example, a bar made of iron and lead (Fe-Pb) it is considered. 
All parameter values for this example are included in Table \ref{tab:ex3}.
\end{example}
\vspace{-0.5cm}
\begin{table}[h]
 \centering
    \caption{Parameter values for Example \ref{ex3}.}
    \begin{tabular}{lc}
Parameter	& Value\\
\hline
$L(m)$ &	5\\
$l(m)$ &	2\\
$k_1 (W/m ^\circ C)$ &	73\\
$k_2  (W/m ^\circ C)$ &	35\\
$\alpha_1^2 (m^2⁄s)$ 	& $0.20451\times10^{-4}$\\
$\alpha_2^2 (m^2⁄s)$	& $0.23673\times10^{-4}$\\
$h (W⁄(m^2 {}^\circ C ))$	& 10\\
$F ({}^\circ C)$	& 150\\
$T_a  ({}^\circ C)$ &	20\\
\hline
    \end{tabular}
    \label{tab:ex3}
\end{table}
%
%
\begin{figure}[h]
    \begin{center}
    \includegraphics[scale=0.6]{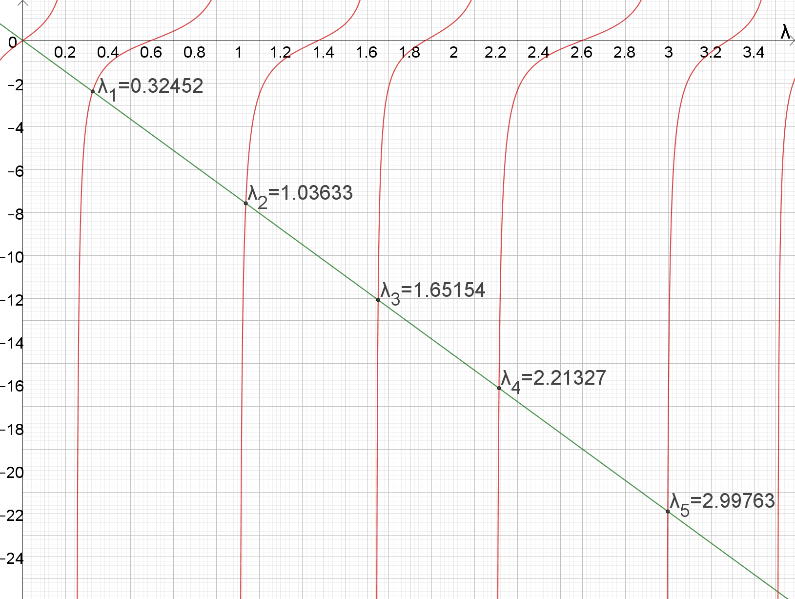}
        \end{center}
        \vspace{-0.5cm}
    \caption{Solutions to Eq. \eqref{eig} associated to  the heat transfer with 
    interface problem \eqref{ec1}-\eqref{FmayorqueTa} for a Fe-Pb bar and modeling parameter values given in Table \ref{tab:ex3} (Example \ref{ex3}
    ).}
    \label{fig:ex3}
\end{figure}
The eigenvalue problem \eqref{eig} in this case becomes
\begin{equation}
    \label{ex3:eig}
-7.3x=\frac{\tan (1.8589 x) + 0.44563 \tan(3x))}{0.44563 -\tan( 1.8589x)  \tan(3x) },
\end{equation}
some of its solutions are shown in Figure \ref{fig:ex3}. Similar results are obtained for different bar compositions. 
%
%

Figure \ref{fig:ex3}
shows some of the solutions of the equations \eqref{ex3:eig}.
These solutions might  not be the first ones, since it could exist discontinuities on the right side of \label{eig2} that do not appear in the plot due, for instance, to the discretization step.
%
\begin{theorem}
The initial-boundary value problem with a solid-solid interface,  described by equations \eqref{ec1}-\eqref{interface2}, has a unique solution of the form
\begin{equation}
\label{Usmasfi}
U(x,t)=U^s (x)+ \varphi (x,t), \,\,0\leq x \leq L,\,  t\geq 0,  
\end{equation}                      
where  $U^s (x)$ is given by the expressions  \ref{Uestacionario} (or
\ref{Uestacionario_mu}-\ref{mu})  and
\begin{equation}
\label{fi}
\varphi(x,t)=
    \begin{cases}
       \varphi_1(x,t),  \quad & 0 \leq x\leq l, \\
       \\
       \varphi_2(x,t),    \quad &      l<x \leq L,
       \end{cases}
\end{equation}
for $t>0$, being
\begin{equation}
\label{fi1}
\hspace{-2cm}\varphi_1(x,t)= \sum_{n=1}^{\infty} C_n \sin(\alpha \lambda_n x) e^{-\lambda_n^2 \alpha_2^2 t},
\end{equation}
\begin{equation}
\label{fi2}
\begin{split}
\varphi_2(x,t)= \sum_{n=1}^{\infty} C_n [k \alpha \cos(\alpha \lambda_n l) \sin( \lambda_n (x-l)) \qquad \\
     + \sin(\alpha \lambda_n l) \cos( \lambda_n (x-l))] 
       e^{-\lambda_n^2 \alpha_2^2 t}, \quad \qquad
       \end{split}
\end{equation}
with
\begin{equation}
\label{cn}
\begin{split}
C_n = 2 (T_a-F) \hspace{5.5cm}\\ \frac{\displaystyle \frac{- \sin(\alpha \lambda_n l)}{\alpha \lambda_n } \displaystyle \frac{\mu h}{k_2} + \cos\left(\alpha \lambda_n l\right) \left(-1+ \frac{\mu h}{k_2}  l\right)+1}
{\alpha \lambda_n  l-\sin(\alpha \lambda_n l)\cos(\alpha \lambda_n l)},
\end{split}
\end{equation}
for $n \in N$ where $\lambda_n$ are the solutions to the equation \eqref{eig} and $\mu$ is defined in \eqref{mu}.
\end{theorem}
\begin{proof}
Equations \eqref{fiXT}, \eqref{alfa}-\eqref{T}, \eqref{k}-\eqref{eig} and the superposition principle lead to Equations \eqref{fi}-\eqref{fi2}. From the initial condition  \eqref{initialc} it follows that
\begin{equation}
    T_a-U^s (x)=\sum_{n=1} ^{\infty} C_n \sin(\alpha \lambda_n  x), \qquad 0 \leq x \leq l,
\end{equation}
where the Fourier coefficients 
$C_n$ are given by 
\begin{equation}
   \displaystyle C_n=
    \frac{ \displaystyle \int_0^l (T_a -U^s(x)) \sin(\alpha \lambda_n x) dx}{ \displaystyle \int_0^l  \sin^2 (\alpha \lambda_n x) dx}.
\end{equation}
Using the dimensionless coefficient $\mu$  defined in \eqref{mu} and after some calculations, the equation \eqref{cn} is obtained.
\end{proof}

\section{Numerical simulations}
\label{numerical}\vspace{-4pt}

The aim of this section is to illustrate the temperature behavior for the heat transfer process given by  \eqref{ec1}-\eqref{interface2}.
 The numerical solutions presented here are obtained by  using a   finite difference  of second order centered in space and forward in time.
  This explicit method is stable and convergent for
 \begin{equation}
     \max\{\alpha_1^2, \alpha_2^2\} < \frac{(\Delta x)^2}{2 \Delta t}, 
 \end{equation}
 where $\Delta x$ and $\Delta t$ are the discretization steps for the space and time, respectively  \cite{Morton2005}.
 
A computational non-parallel scheme was programmed in Matlab.  A regular partition is considered in space and time to discretize the equations,  taking   $\Delta x=0.01 m.$ and $\Delta t= 0.1 s.$ so that
 $\frac{(\Delta x)^2}{2 \Delta t}=5\times10^{-4}$ which is greater that all possible thermal diffusivity coefficients $\alpha_1^2$, $\alpha_2^2$ considered for this work (see Table \ref{thermal_prop}).
The simulations are obtained in few seconds when using an Intel(R) Core(TM) i7-6700K 4.GHz machine.
\begin{example}
\label{numerical_ex1}
Consider the problem described by the equations 
\eqref{ec1}-\eqref{FmayorqueTa} with $L = 1 \, m$,
where the solid-solid  interface is located at  $l=0.3\,   m$,  the heat transfer coefficient is $h=10\, W⁄(m^2 {}^\circ C )$, $T_a=25^\circ C$,  and the thermal source is $F=100^\circ C$.
\end{example}
Figures \ref{fig:Tempxl}-\ref{fig:TempxL} show the plots for the temperature profiles at the interface $x=l$ and at the right boundary $x=L$, respectively,  for a bar composed by different pairs of materials where the material at the left side of the bar is  Pb (top) and Ag (bottom).
\begin{figure}
\begin{center}
    \includegraphics[scale=0.13]{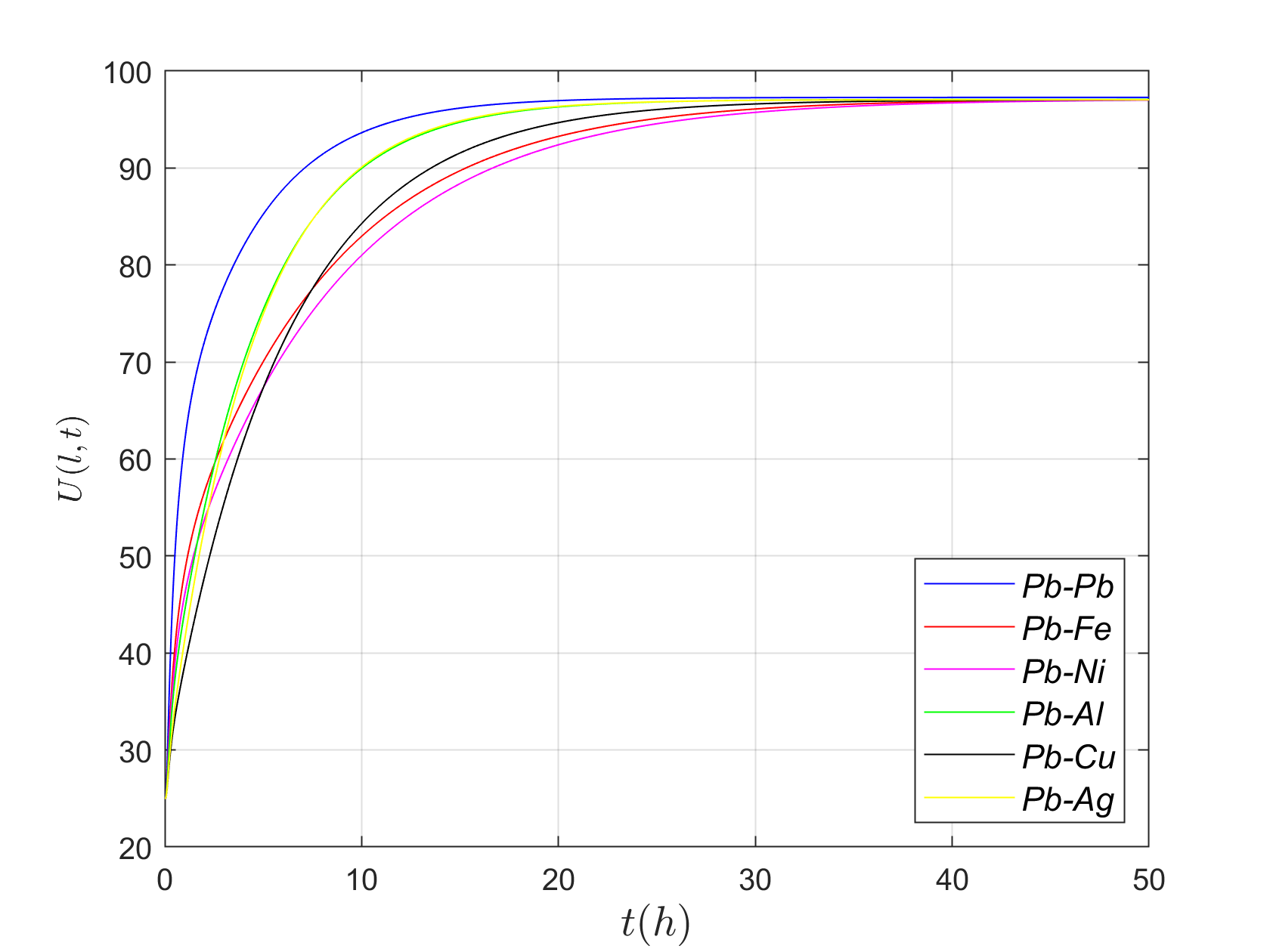}\\
    \includegraphics[scale=0.13]{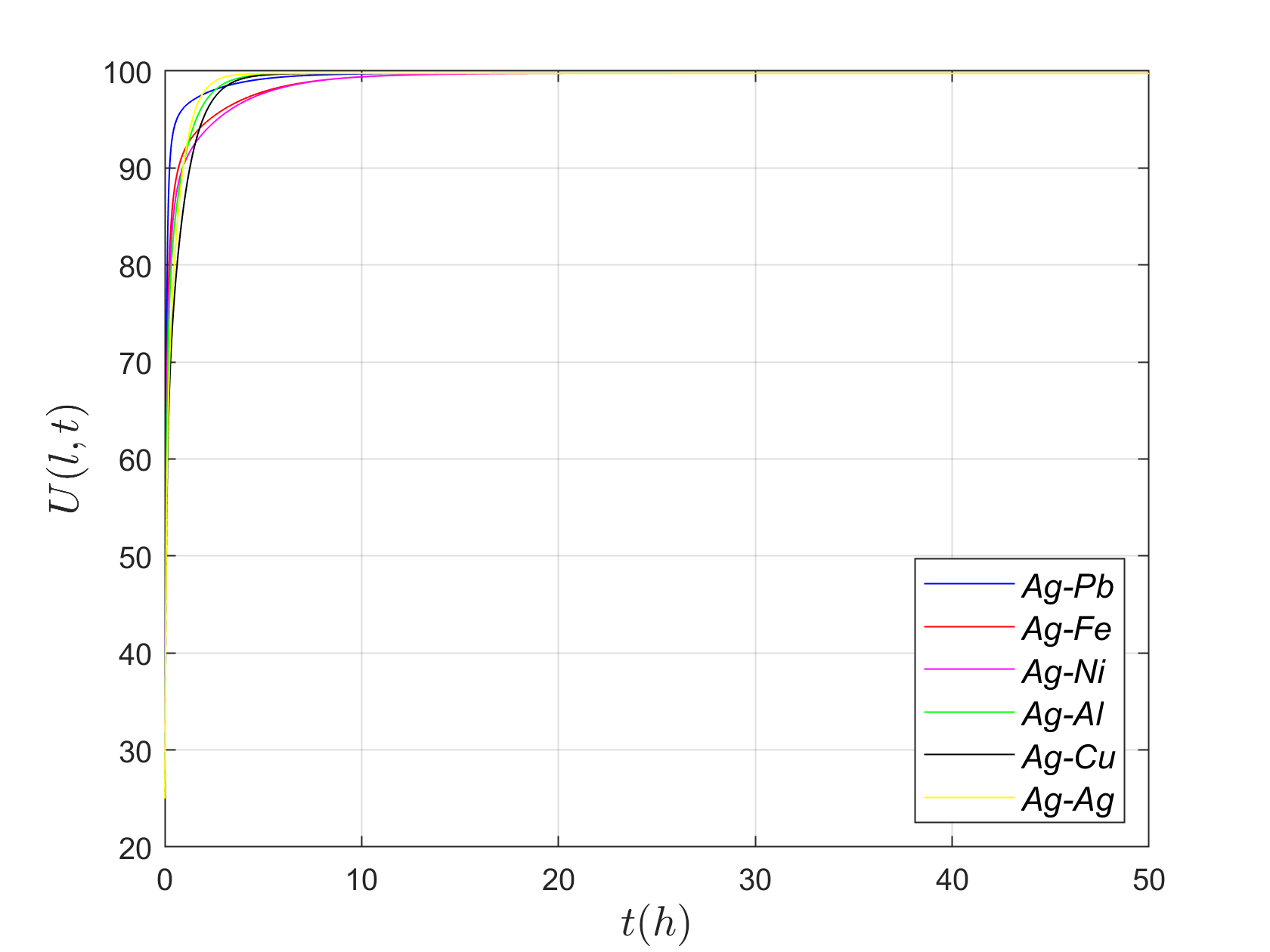}
 \end{center}
 \vspace{-0.5cm}
 \caption{ \label{fig:Tempxl}Temperature profiles at the interface point ($l = 0.3 \, m$) for the Example \ref{numerical_ex1} where the material at the left side of the bar is  Pb (top) and Ag (bottom).}
    \end{figure}
From these figures, it can be seen that $U(l,t)>U(L,t)$ for all $t>0$, that agrees with the analytical solution given in \eqref{fi2}.
It is also observed that in all cases it requires some hours to achieve the steady-state,  and it is reached earlier when more diffusive materials are involved. 
These  observations are also consistent with the analytical solution, since the transient terms of the solution, \eqref{fi1}-\eqref{fi2} (and \eqref{T1T2}, \eqref{alfa}, \eqref{T}) decay exponentially with the diffusivity coefficients which are  of the order of $ 10^{ -4}$.
\begin{figure}
\begin{center}
    \includegraphics[scale=0.13]{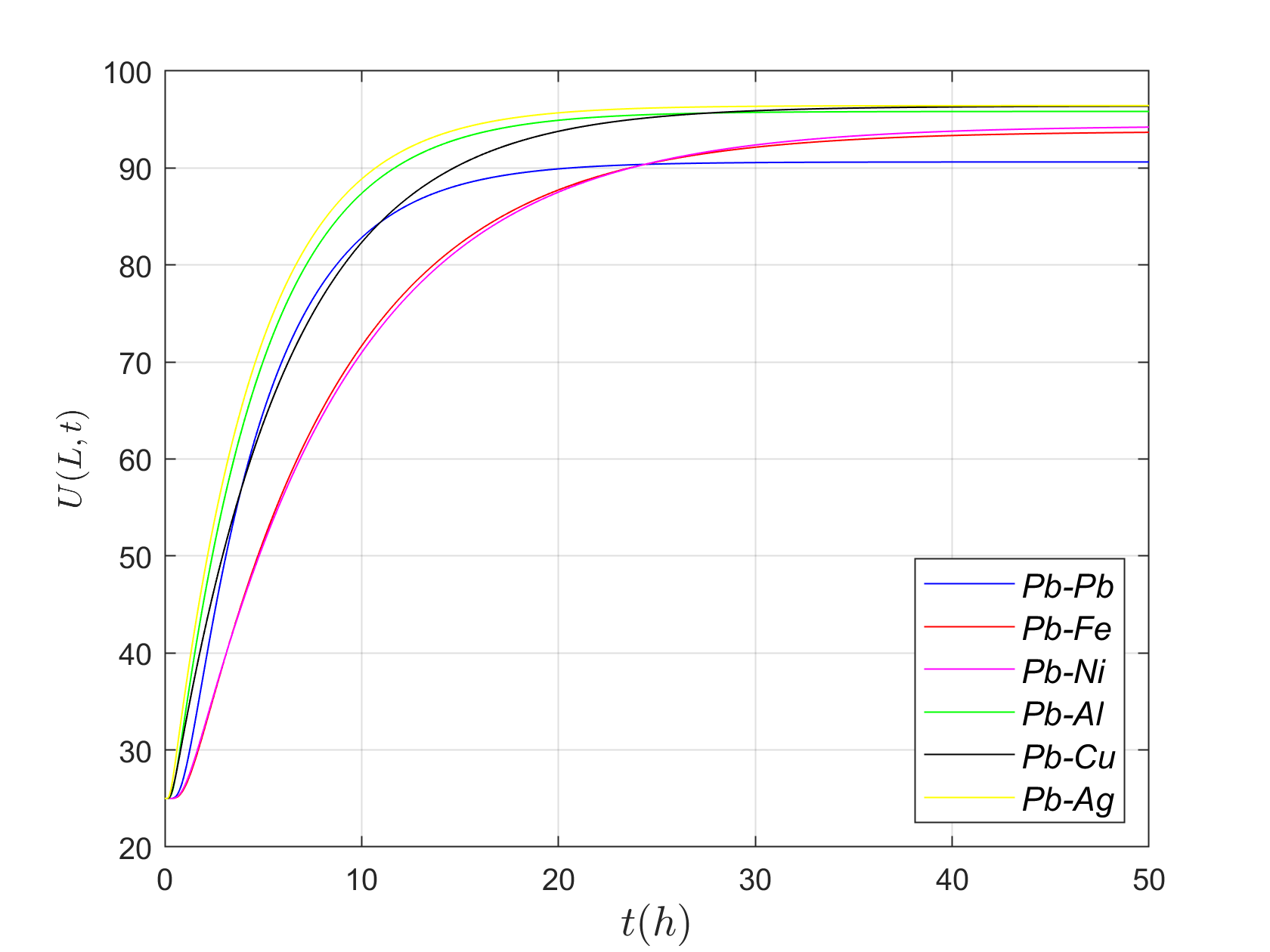}\\
        \includegraphics[scale=0.13]{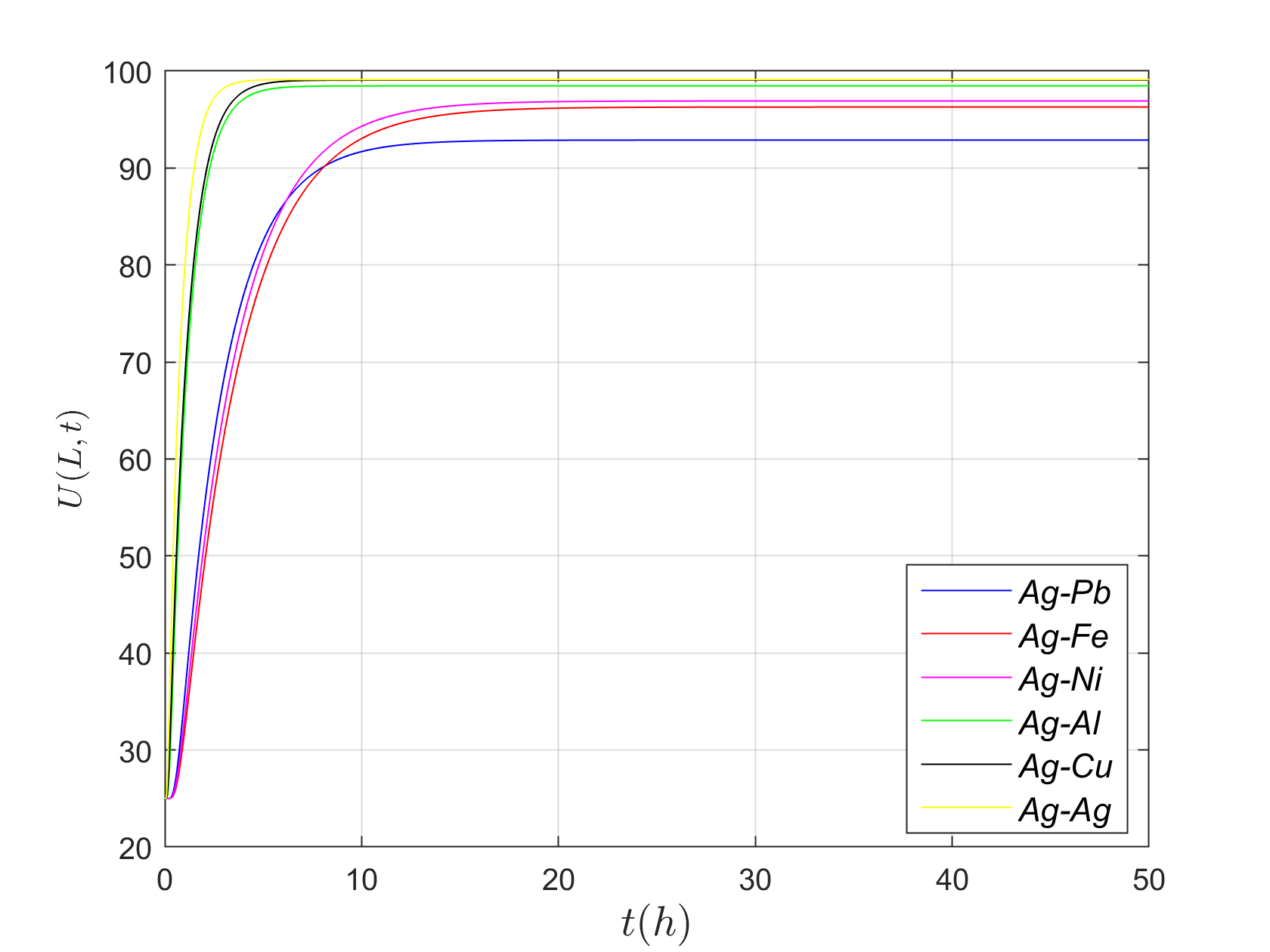}
    \end{center}
   \vspace{-0.5cm}
 \caption{Temperature profiles at the free-end ($x = L$)   for the Example \ref{numerical_ex1} where the material at the left side of the bar is  Pb (top) and Ag (bottom).}
    \label{fig:TempxL}
    \end{figure}
 \begin{figure}
\begin{center}
 \includegraphics[scale=0.13]{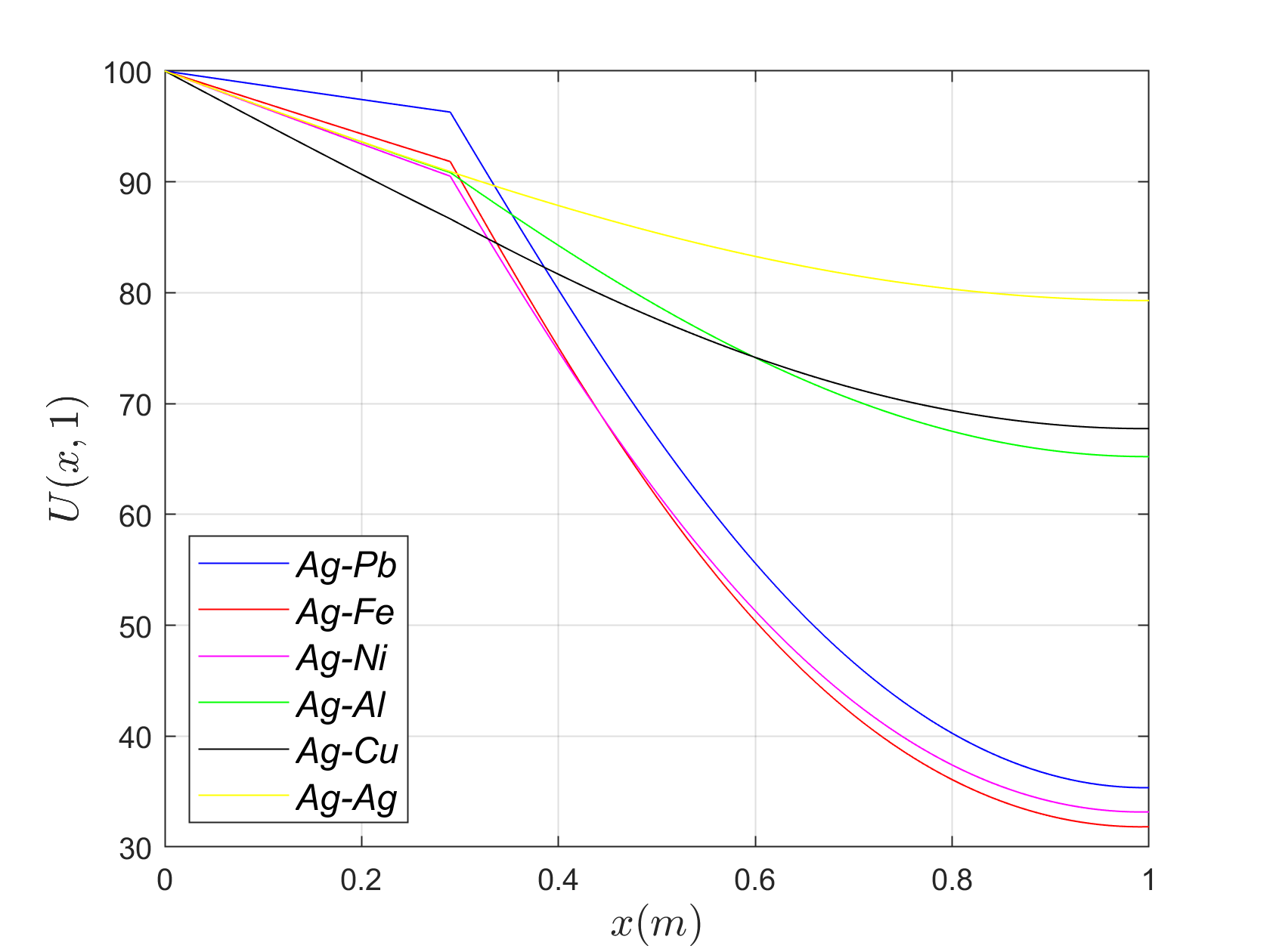}\\
 \includegraphics[scale=0.13]{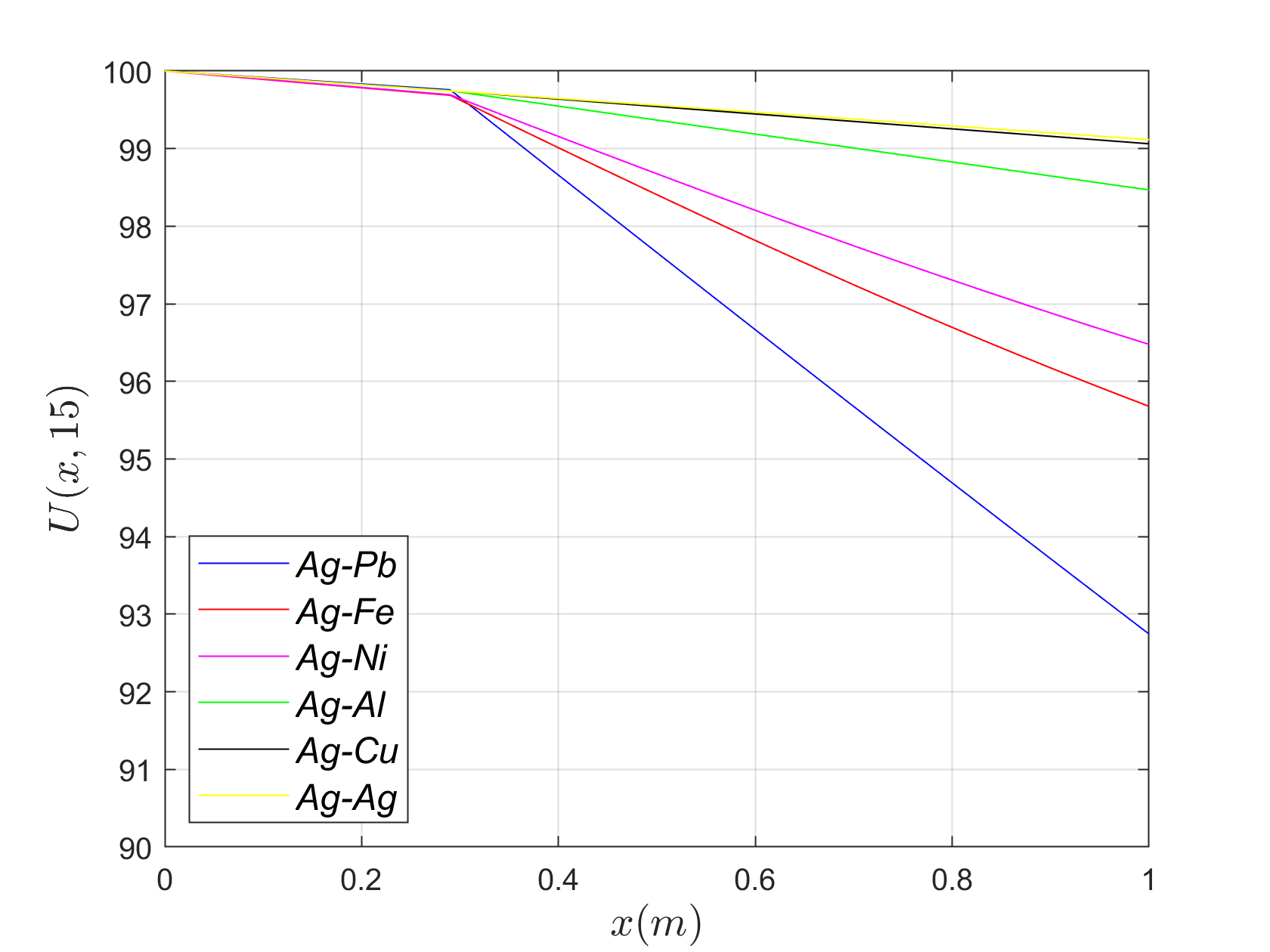}
         \end{center}
  \vspace{-0.5cm}
         \caption{Bar temperature at $t=1h$ (top) and $t=15h$ (bottom) for the Example \ref{numerical_ex1}.}
         \label{fig:Temph}
  \end{figure}
     \begin{figure}
 \begin{center}
\includegraphics[scale=0.25]{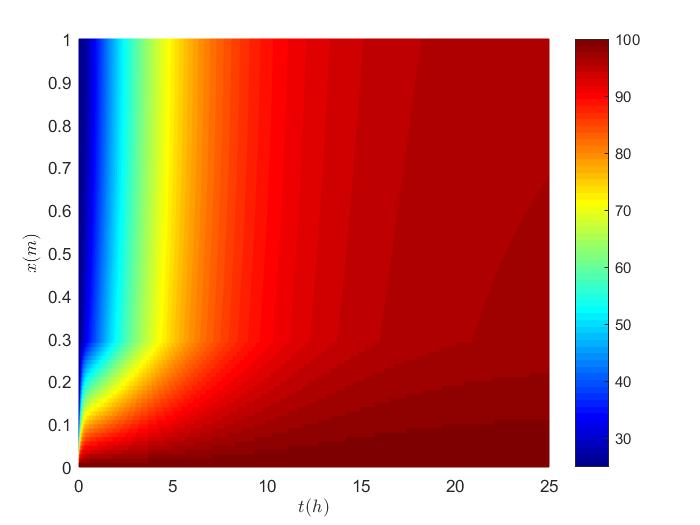}\\
\includegraphics[scale=0.25]{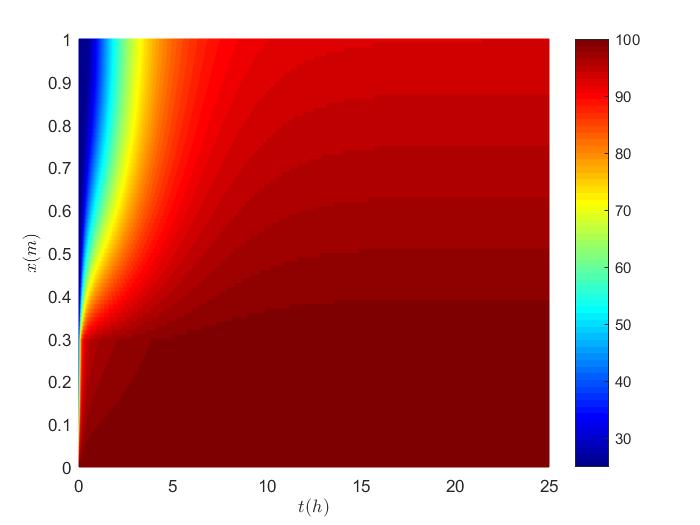}
 \end{center}
  \vspace{-0.5cm}
 \caption{Temperature as function of space and time for the Example \ref{numerical_ex1}.  Pb-Ag (top),  Ag-Pb (bottom).}
\label{fig:Temptx}
    \end{figure}
    
In Figure \ref{fig:Temph} temperature profiles  on the bar at $t=1h$  and at $t=15h$ are shown.
Note that for the latter, the curves resemble piecewise linear functions, which correspond to the steady-state as shown in the analytical formula given in \eqref{Uestacionario} and it is illustrated in Figure \ref{fig:stationary}. The slopes depend  on a particular combination of the conductivity values of the materials, the location of the interface and the source and room temperatures. This also agree with the analytical solution given in  \eqref{Uestacionario}, \eqref{Usmasfi}-\eqref{fi2} since the transitory terms approaches zero with time.

 Finally, in Figure \ref{fig:Temptx} the temperatures for Pb-Ag (top)  and Ag-Pb (bottom) as functions of space and time are  plotted, where the horizontal axis represents the time in hours and the vertical axis represents the distance from the left boundary in meters. 
 That is, for a fixed value of $ t $, the temperature distribution of the bar at that time can be seen vertically, from the left edge  $x = 0$ (bottom line of the graph) to the right one, $x = L$ (top line of the graph).
 On the other hand, taking a fixed value of $ x $, one can see the evolution of the temperature at that point by looking at the corresponding horizontal line.
 Notice, in both cases, a change in the  temperature behavior at the interface point ($x=0.3m$). Moreover, for $x \leq 0.3m$, the plot on bottom (Ag-Pb) shows that the temperature achieves higher values in a shorter period of time than for the corresponding one for Pb-Ag (top) under the same conditions. This observation is physically consistent  to the fact that Ag is a more diffusive material than Pb.  The materials for this example were chosen so that the differences in the behavior of the temperature function can be easily observed due to the large difference in their respective thermal diffusivities.
 
\section{Conclusion}
\label{S3} \vspace{-4pt}

In this work, the solution to a heat transfer problem along a bar with a solid-solid interface is considered. This study pursues to provide a theoretical basement that can help to gain insight into the effect of interfaces on heat transfer processes, from the mathematical point of view. A perfect assembly between the two parts are considered, so that differences between the analytical solution and experimental measurements will provide an amount of thermal dissipation between the two materials, that would be useful to model tension and roughness at the interface as well as solid-solid thermal resistance.  
The problem is described by an initial value parabolic partial differential equation with interface and Dirichlet and Robin boundary conditions. The analytical expression for the solution is derived where the steady-state form is explicitly included. The transient part of the solution is obtained which depends on the solution of a Sturm-Liouville problem. 
The existence of an infinite number of solutions to the eigenvalue problem is demonstrated and it is the most important  result of this work. Also, an illustrative example is included.

Numerical simulations are conducted by using an explicit finite difference scheme  where its convergence and stability properties are discussed.
Numerical results are consistent with analytical solutions and physical interpretations. 

Future works might include, among others, the study of mathematical models for the thermal behavior at the interface and how the imperfections or roughness at the solid-solid interface can change the temperature distribution at the bar. Also, extensions to 2D and 3D analysis and/or the problem for two or more interfaces can be conducted.

\vspace{10pt} \noindent
{\bf Acknowledgements:}  The research was supported by the
Universidad de San Mart\'in, Universidad Austral and, in the case of the first and third authors, by SOARD/AFOSR  (Grant FA9550-18-1-0523).
 The second author acknowledges support from European Union's Horizon 2020 Research and Innovation Programme under the Marie Sklodowska-Curie Grant Agreement No. 823731 CONMECH and by the Project PIP No. 0275 from CONICET-UA, Rosario, Argentina.

\end{document}